\pdfoutput=1
\providecommand{\noopsort}[1]{} 

\documentclass{amsart}

\usepackage[mathscr]{eucal}
\usepackage{amssymb}
\usepackage[dvipsnames]{xcolor} 
\usepackage[normalem]{ulem}
\usepackage{amsthm}
\usepackage{bbold}
\usepackage{comment}
\usepackage{enumitem}
\usepackage{amsmath}
\usepackage{tikz-cd}

\usepackage[unicode]{hyperref} 
\hypersetup{
    colorlinks, linkcolor=Blue,
    citecolor=Blue, urlcolor=Blue,
    breaklinks=true
}

\usepackage[nameinlink,capitalise,noabbrev]{cleveref}

\numberwithin{equation}{section}
\usepackage[all]{xy}
\xyoption{line}
\usepackage{graphicx}
\usepackage{mathtools}

\swapnumbers 

\newtheorem{Thm}[equation]{Theorem}
\newtheorem*{Thm*}{Theorem}
\newtheorem*{Q*}{Question}
\newtheorem{Prop}[equation]{Proposition}
\newtheorem{Lem}[equation]{Lemma}
\newtheorem{Cor}[equation]{Corollary}

\theoremstyle{remark}
\newtheorem{Def}[equation]{Definition}
\newtheorem*{Def*}{Definition}

\newtheorem{Not}[equation]{Notation}
\newtheorem{Exa}[equation]{Example}
\newtheorem{Hyp}[equation]{Hypothesis}
\newtheorem{Rec}[equation]{Recollection}
\newtheorem{Rem}[equation]{Remark}

\tikzset{
    labelrotatebelow/.style={anchor=north, rotate=90, inner sep=1.0mm}
}
\tikzset{
    labelrotateabove/.style={anchor=south, rotate=90, inner sep=1.0mm}
}

\newcommand{\nc}{\newcommand}
\nc{\rnc}{\renewcommand}
\nc{\dmo}{\DeclareMathOperator}

\rnc{\emptyset}{\varnothing}
\nc{\Mid}{\,\big|\,}
\nc{\SET}[2]{\big\{\,#1\Mid#2\,\big\}} 

\dmo{\Ker}{Ker}
\dmo{\End}{End}
\dmo{\Hom}{Hom}
\dmo{\im}{im}
\dmo{\id}{id}
\nc{\inv}{^{-1}}

\nc{\hook}{\hookrightarrow}
\nc{\adjto}{\rightleftarrows}
\nc{\xra}{\xrightarrow}
\nc{\xhook}{\xhookrightarrow}
\nc{\xtwohead}[1]{\overset{#1}{\twoheadrightarrow}}

\dmo{\gen}{gen}
\dmo{\cl}{cl}
\dmo{\loc}{loc}
\dmo{\con}{con}
\dmo{\stmod}{stmod}
\dmo{\StMod}{StMod}
\dmo{\Proj}{Proj}
\dmo{\SH}{SH}
\nc{\SHc}{{\SH^c}}
\nc{\SHp}{{\SH_{(p)}}}
\nc{\SHcp}{{\SH^c_{(p)}}}

\dmo{\Ann}{Ann}
\dmo{\WeakAss}{WeakAss}
\dmo{\Ass}{Ass}
\dmo{\modname}{mod}
\dmo{\Mod}{Mod}
\nc{\MMod}[1]{\Mod(#1)}
\nc{\mmod}[1]{\modname(#1)}
\dmo{\Der}{D}
\dmo{\Derqc}{D_{qc}}
\dmo{\Derb}{D^{b}}
\dmo{\Derperf}{D^{perf}}
\nc{\Rder}{\mathrm{R}}
\nc{\Lder}{\mathrm{L}}

\dmo{\cone}{cone}
\dmo{\supp}{supp}
\dmo{\Supp}{Supp}
\nc{\SuppBIK}{\Supp_{\textup{BIK}}}
\nc{\kos}[2]{{#1/\!\!/#2}}
\dmo{\Cosupp}{Cosupp}
\nc{\haux}{\mathrm{h}}
\nc{\Supph}{\Supp^\haux}
\nc{\Spch}{\Spc^\haux}
\dmo{\Spec}{Spec}
\dmo{\Spech}{Spec^h}
\dmo{\Spc}{Spc}

\nc{\ideal}[1]{\langle #1\rangle}
\dmo{\thick}{thick}
\nc{\thickt}[1]{\thick_\otimes\ideal{#1}}
\nc{\Loc}[1]{\operatorname{Loc}\ideal{#1}}
\nc{\Loco}[1]{\operatorname{Loc}_{\otimes}\hspace{-0.3ex}\ideal{#1}}
\nc{\Colocid}[1]{\operatorname{Colocid}\ideal{#1}}

\dmo{\rmH}{H}
\nc{\HFp}{{\rmH \hspace{-0.15em}\bbF_{\hspace{-0.1em}p}}}

\nc{\frakm}{\mathfrak{m}}
\nc{\frakp}{\mathfrak{p}}
\nc{\frakq}{\mathfrak{q}}
\nc{\fraka}{\mathfrak{a}}
\nc{\frakb}{\mathfrak{b}}
\nc{\bbN}{\mathbb{N}}
\nc{\bbZ}{\mathbb{Z}}
\nc{\bbF}{\mathbb{F}}
\nc{\cal}[1]{\mathcal{#1}}
\nc{\cat}[1]{\mathscr{#1}}
\nc{\unit}{\mathbb{1}}

\nc{\ie}{{i.e.}, }
\nc{\cf}{{cf.~}}
\nc{\aka}{{a.\,k.\,a.}\ }
\nc{\eg}{{\sl e.\,g.}}
\nc{\loccit}{{\sl loc.\,cit.}\ }

\dmo{\Ind}{Ind}
\dmo{\Map}{Map}
\dmo{\Ho}{Ho}
\dmo{\CAlg}{CAlg}
\dmo{\Fun}{Fun}
\nc{\Prst}{{\mathsf{Pr}^{L}_{\mathrm{st}}}}
\nc{\Sp}{\mathscr{S}p}

\newcounter{enum-resume-hack}

\begin{document}

\title[Convexity in tensor triangular geometry]{Convexity in tensor triangular geometry}

\author{Changhan Zou}

\date{August 7, 2025}

\address{Changhan Zou, Mathematics Department, UC Santa Cruz, 95064 CA, USA}
\email{czou3@ucsc.edu}
\urladdr{\href{https://people.ucsc.edu/~czou3}{https://people.ucsc.edu/$\sim$czou3}}

\begin{abstract}
We classify the dualizable localizing ideals of rigidly-compactly generated tt-$\infty$-categories that are cohomologically stratified. By definition, these are the localizing ideals that are dualizable with respect to the Lurie tensor product. We prove that these ideals correspond to the convex subsets of the Balmer spectrum. More generally, we establish this classification for categories which are locally cohomologically stratified and whose Balmer spectrum is noetherian. The classification thus applies to many categories arising in algebra and topology, including derived categories of noetherian schemes. Our result generalizes, and is motivated by, a recent theorem of Efimov which establishes this classification for derived categories of commutative noetherian rings.
\end{abstract}

\maketitle

\vspace{-3ex}
{
\hypersetup{linkcolor=black}
\tableofcontents
}
\vspace{-3ex}

\section{Introduction}
Hilbert's Nullstellensatz states that there is a one-to-one correspondence between the closed subsets of an affine variety and the radical ideals of its coordinate ring. This type of ideal--subset correspondence is common in mathematics. For example, inspired by the thick subcategory theorem \cite{HopkinsSmith98} in chromatic homotopy theory, Hopkins--Neeman \cite{Hopkins87, Neeman92a} classified the thick ideals of the derived category of perfect complexes over a commutative noetherian ring $R$. This provides an inclusion-preserving bijection
\[
\{\text{thick ideals of }\Derperf(R)\} \xra{\sim} 
\{\text{specialization closed subsets of }\Spec(R)\}.
\]
This classification was then extended to the unbounded derived category $\Der(R)$ by Neeman \cite{Neeman92a}: There is an inclusion-preserving bijection
\[
\{\text{localizing ideals of }\Der(R)\} \xra{\sim} \{\text{subsets of }\Spec(R)\}.
\]
The connection between these two results is that there is a one-to-one correspondence between the collection of thick ideals of $\Derperf(R)$ and the collection of localizing ideals of $\Der(R)$ that are generated by perfect complexes. Recently, Efimov \cite{Efimov24} refined Neeman's classification in his study of extensions of compactly generated stable $\infty$-categories. More precisely, Efimov showed that the localizing ideals of $\Der(R)$ that are dualizable with respect to the Lurie tensor product correspond to the \emph{convex} subsets of $\Spec(R)$. Putting all these results in one picture, we obtain:
\[
\begin{tikzcd}
\{\text{localizing ideals of }\Der(R)\} \arrow[r, "\sim", "\text{Neeman}"']                                                                                                       & \{\text{subsets of }\Spec(R)\}                                        \\
\left\{\begin{gathered} \text{dualizable localizing} \\ \text{ideals of }\Der(R) \end{gathered} \right\} \arrow[u, hook'] \arrow[r, "\sim", "\text{Efimov}"']                     & \left\{\begin{gathered} \text{convex subsets} \\ \text{of }\Spec(R) \end{gathered} \right\} \arrow[u, hook']                \\
\left\{\begin{gathered} \text{localizing ideals of }\Der(R) \\ \text{ generated by perfect complexes} \end{gathered} \right\} \arrow[u, hook'] \arrow[r, "\sim", "\text{Hopkins}"'] &\left\{\begin{gathered}
\text{specialization closed} \\
\text{subsets of }\Spec(R)
\end{gathered} \right\}.  \arrow[u, hook']
\end{tikzcd}
\]

The purpose of this article is to lift this picture to the realm of tensor triangular geometry. In \cite{Balmer05a}, Balmer constructed, out of any essentially small tensor triangulated category $\cat K$, a topological space $\Spc(\cat K)$, which is now called the Balmer spectrum. This space comes with a universal notion of support and it captures the global structure of $\cat K$ in the sense that the support induces an inclusion-preserving bijection
\[
\{\text{thick ideals of }\cat K\} \xra{\sim} \{\text{specialization closed subsets of }\Spc(\cat K)\}.
\]
Here we assume that $\cat K$ is rigid and that $\Spc(\cat K)$ is noetherian for expositional simplicity. This abstract classification theorem unifies major classification theorems from algebraic topology, algebraic geometry, and modular representation theory \cite{Thomason97, BensonCarlsonRickard97, HopkinsSmith98}. For instance, the classification of Hopkins reads: $\Spc(\Derperf(R)) \cong \Spec(R)$.

The category $\cat K$ often arises as the subcategory $\cat T^c$ of compact objects of a rigidly-compactly generated tensor triangulated category $\cat T$. In favorable circumstances, the spectrum $\Spc(\cat T^c)$ also controls the global structure of the bigger category $\cat T$. Namely, there exists an extension of the universal support theory to arbitrary objects of $\cat T$, and this provides a bijection
\begin{equation}\label{eq:strat-intro}
\{\text{localizing ideals of }\cat T\} \xra{\sim} \{\text{subsets of }\Spc(\cat T^c)\}.
\end{equation}
When this occurs, we say that $\cat T$ is \emph{stratified}. The study of stratified categories arose in the work of Hovey--Palmieri--Strickland \cite{HoveyPalmieriStrickland97} and Benson--Iyengar--Krause \cite{BensonIyengarKrause11b}. More recently, it has been systematically developed in the context of tensor triangular geometry by Barthel--Heard--Sanders \cite{BarthelHeardSanders23b}. For example, Neeman's theorem can be rephrased as saying that $\Der(R)$ is stratified. Under stratification, Balmer's classification of thick ideals can be interpreted as the bijection
\begin{equation}\label{eq:balmer-intro}
\left\{\begin{gathered} \text{localizing ideals of }\cat T\text{ generated} \\ \text{by compact objects of }\cat T \end{gathered} \right\} \xra{\sim} \left\{\begin{gathered}
\text{specialization closed} \\
\text{subsets of }\Spc(\cat T^c)
\end{gathered} \right\}.
\end{equation}
In practice, $\cat T$ is frequently observed as the homotopy category of a symmetric monoidal presentable stable $\infty$-category $\cat C$. When $\cat C$ is stratified, every localizing ideal of $\cat C$ is itself a presentable stable $\infty$-category. We can thus consider those which are dualizable with respect to the Lurie tensor product. This leads to the following:
\begin{Q*}
Let $\cat C$ be a rigidly-compactly generated tt-$\infty$-category that is stratified. Do the dualizable localizing ideals of $\cat C$ correspond to the convex subsets of $\Spc(\cat C^c)$?
\end{Q*}
\noindent We provide an affirmative answer to this question for a large class of categories:
\begin{Thm}\label{thm:intro}
Let $\cat C$ be a rigidly-compactly generated tt-$\infty$-category with $\Spc(\cat C^c)$ noetherian. If $\cat C$ is locally cohomologically stratified then the tensor triangular support induces a bijection
\[
\{\text{dualizable localizing ideals of }\cat C\} \xra{\sim} \{\text{convex subsets of }\Spc(\cat C^c)\}.
\]
Moreover, in this case, the dualizable localizing ideals of $\cat C$ are themselves compactly generated.
\end{Thm}
\noindent This is proved as \cref{thm:main}; see \cref{def:loc-coh-strat} for the definition of locally cohomologically stratified. We therefore obtain the following picture:
\[
\begin{tikzcd}
\{\text{localizing ideals of }\cat C\} \arrow[r, "\sim", "\eqref{eq:strat-intro}"']                                                                                                       & \{\text{subsets of }\Spc(\cat C^c)\}                                        \\
\{\text{dualizable localizing } \text{ideals of }\cat C\} \arrow[u, hook'] \arrow[r, "\sim", "\eqref{thm:intro}"']                     & \{\text{convex subsets of }\Spc(\cat C^c)\} \arrow[u, hook']                \\
\left\{\begin{gathered} \text{localizing ideals of }\cat C\text{ generated} \\ \text{by compact objects of }\cat C \end{gathered} \right\} \arrow[u, hook'] \arrow[r, "\sim", "\eqref{eq:balmer-intro}"'] &\left\{\begin{gathered}
\text{specialization closed} \\
\text{subsets of }\Spc(\cat C^c)
\end{gathered} \right\}.  \arrow[u, hook']
\end{tikzcd}
\]
Every cohomologically stratified category satisfies the hypotheses of our theorem. This provides a large class of examples. On the other hand, our theorem also applies to derived categories of noetherian schemes since they are locally cohomologically stratified. This extends the affine case established by Efimov in \cite[Theorem~3.10]{Efimov24}.

The proof of our theorem is inspired by that of Efimov, although not all of his techniques transfer to our more general setting; see \cref{rem:proof}. One crucial input is Efimov's generalization of the Neeman--Thomason theorem, from compactly generated categories to dualizable categories; see \cref{thm:neeman--thomason}. Without this input, one can only show that the convex subsets correspond to the localizing ideals that are themselves compactly generated, which is a strictly weaker statement. Another aspect of our proof is using the local cohomological stratification condition to compute the support of objects explicitly via the results of \cite{BensonIyengarKrause08}.

\subsection*{Acknowledgements}
The author is grateful to Beren Sanders for all kinds of support, without whom this project could not be finished. He also thanks Xu Gao and Jiacheng Liang for helpful conversations. Finally, he would like to thank MPIM for organizing the \emph{Workshop on Dualizable Categories and Continuous K-theory} during which he learned about the work of Efimov.

\medskip
\section{Higher tt-categories}
In this section, we set the stage by reviewing necessary background on higher categories in tensor triangular geometry. The standard references are \cite{Lurie_HTT, Lurie_HA}.

\begin{Not}
For any $\infty$-category $\cat C$, we write $\Map_{\cat C}$ for the mapping space and $\Ho(\cat C)$ for its homotopy category. If $\cat C$ admits a symmetric monoidal structure, we write $\CAlg(\cat C)$ for the $\infty$-category of commutative algebras in $\cat C$.
\end{Not}

\begin{Rec}
Recall that $\Prst$ denotes the $\infty$-category of presentable stable $\infty$-categories and colimit-preserving functors. The following terminology comes from \cite[Section~5]{BCHNPS24}:
\end{Rec}

\begin{Def}
A \emph{big tt-$\infty$-category} is an object $\cat C \in \CAlg(\Prst)$. In other words, a big tt-$\infty$-category is a symmetric monoidal presentable stable $\infty$-category whose tensor commutes with colimits in each variable. A morphism of big tt-$\infty$-categories is a morphism in $\CAlg(\Prst)$; that is, a colimit-preserving symmetric monoidal functor.
\end{Def}

\begin{Rem}
For $\cat C \in \CAlg(\Prst)$ we write $\Mod_{\cat C} \coloneqq \Mod_{\cat C}(\Prst)$ for the $\infty$-category of $\cat C$-modules. The relative Lurie tensor product $\otimes_{\cat C}$ makes $\Mod_{\cat C}$ into a closed symmetric monoidal $\infty$-category with unit $\cat C$ and internal hom $\Fun^L_{\cat C}$.
\end{Rem}

\begin{Def}
We say that a $\cat C$-module $\cat M$ is \emph{dualizable over} $\cat C$ if $\cat M$ is a dualizable object in $\Mod_{\cat C}$. In particular, we say that $\cat M \in \Prst \simeq \Mod_{\Sp}$ is \emph{dualizable} when it is dualizable over the $\infty$-category of spectra.
\end{Def}

\begin{Exa}
Every compactly generated $\cat M \in \Prst$ is dualizable. Conversely, every dualizable $\cat M$ is a retract of a compactly generated category. See \cite[D.7.3.1]{Lurie18}.
\end{Exa}

\begin{Not}
For $\cat C \in \Prst$ we write $\cat C^c$ for the full subcategory of compact objects and for $\cat C \in \CAlg(\Prst)$ we write $\cat C^d$ for the full subcategory of dualizable objects. The following is a special case of \cite[Definition~4.5 and Definition~4.34]{Ramzi24}:
\end{Not}

\begin{Def}
A big tt-$\infty$-category $\cat C \in \CAlg(\Prst)$ is \emph{locally rigid} if $\cat C$ is dualizable and the multiplication $\cat C \otimes \cat C \to \cat C$ is an internal left adjoint in $\Mod_{\cat C \otimes \cat C}$. We say that $\cat C$ is \emph{rigid} if it is locally rigid and its tensor unit $\unit$ is compact.
\end{Def}

\begin{Rem}\label{rem:locally-rigid}
Suppose that $\cat C \in \CAlg(\Prst)$ is compactly generated. Then $\cat C$ is locally rigid if and only if $\cat C^c \subseteq \cat C^d$. This follows from \cite[Example~4.6]{Ramzi24}. Hence, $\cat C$ is rigid if and only if $\cat C^c \subseteq \cat C^d$ and $\unit \in \cat C^c$. This is the case exactly when $\cat C^c = \cat C^d$, by \cite[Lemma~4.50]{Ramzi24}.
\end{Rem}

\begin{Rem}\label{rem:term}
The terminology above is compatible with the terminology commonly used in the literature on tensor triangular geometry. Note that a stable $\infty$-category $\cat C$ is presentable if and only if the triangulated category $\Ho(\cat C)$ is well generated, and that $\cat C$ is compactly generated as an $\infty$-category if and only if $\Ho(\cat C)$ is compactly generated as a triangulated category; see \cite[Section~1.3]{Efimov24} and \cite[Lemma~A.8 and Remark~A.9]{DellAmbrogioMartos24}. Moreover, it follows from \cref{rem:locally-rigid} that a big tt-$\infty$-category $\cat C$ is rigid and compactly generated if and only if $\Ho(\cat C)$ is a rigidly-compactly generated tt-category\footnote{In the terminology of \cite{HoveyPalmieriStrickland97}, this is equivalent to $\Ho(\cat C)$ being a unital algebraic stable homotopy category. More generally, $\cat C$ is locally rigid and compactly generated if and only if $\Ho(\cat C)$ is an algebraic stable homotopy category.} in the usual sense of \cite{BalmerFavi11}. Furthermore, a functor $\cat C \to \cat D$ is a morphism of big tt-$\infty$-categories if and only if the induced functor $\Ho(\cat C) \to \Ho(\cat D)$ is a geometric functor in the sense of \cite{BarthelHeardSanders23b}.
\end{Rem}

\begin{Rem}
We record the following two lemmas for later use, whose special cases are implicitly used in the proof of \cite[Theorem~3.10]{Efimov24}.
\end{Rem}

\begin{Lem}\label{lem:fully-faithful}
Let $\cat C \to \cat D$ be a morphism in $\CAlg(\Prst)$. If $\cat D$ is dualizable and $\cat C$ is locally rigid, then the relative tensor product $-\otimes_{\cat C}\cat D \colon \Mod_{\cat C} \to \Mod_{\cat D}$ preserves fully faithful morphisms.
\end{Lem}

\begin{proof}
By \cite[Proposition~4.17]{Ramzi24}, $\cat D$ is dualizable over $\cat C$, so we have an equivalence $-\otimes_{\cat C}\cat D \simeq \Fun^L_{\cat C}(\Fun^L_{\cat C}(\cat D, \cat C), -)$. The latter preserves fully faithful morphisms, which completes the proof.
\end{proof}

\begin{Lem}\label{lem:dualizability}
Consider $\cat C \in \CAlg(\Prst)$ and $\cat C$-modules $\cat M$ and $\cat N$. If $\cat C$, $\cat M$, and $\cat N$ are dualizable, then so is $\cat M\otimes_{\cat C}\cat N$.
\end{Lem}

\begin{proof}
Recall from \cite[Section~4.4]{Lurie_HA} that the relative tensor product is computed as the colimit of the geometric realization of a two-sided Bar construction in which the tensors are over $\Sp$. Therefore, $\cat M\otimes_{\cat C}\cat N$ is a colimit of dualizable categories and hence is itself dualizable by \cite[Proposition~1.65]{Efimov24}.
\end{proof}

\section{Stratification in tt-geometry}
We now briefly recall the theory of stratification in tensor triangular geometry. More details can be found in \cite{BarthelHeardSanders23b}. We take for granted some familiarity with the Balmer spectrum of an essentially small rigid tt-category as in \cite{Balmer05a}.

\begin{Hyp}
In this section, $\cat T$ will denote a rigidly-compactly generated tt-category whose spectrum $\Spc(\cat T^c)$ is weakly noetherian. We will occasionally assume that there is an underlying (rigidly-compactly generated) tt-$\infty$-category $\cat C$ and write $\cat T = \Ho(\cat C)$ to indicate this.
\end{Hyp}

\begin{Rec}
Recall from \cite{BarthelHeardSanders23b} that there exists an idempotent object $g_{\cat P}$ for each $\cat P \in \Spc(\cat T^c)$ such that
\[
g_{\cat P} \otimes g_{\cat Q} \neq 0 \text{ for } \cat P \neq \cat Q.
\]
Moreover, the tensor triangular support of an object $t \in \cat T$ is defined to be the set
\[
\Supp(t) \coloneqq \SET{\cat P \in \Spc(\cat T^c)}{t \otimes g_{\cat P} \neq 0}.
\]
For any localizing ideal $\cat L$, we set
\[
\Supp(\cat L) \coloneqq \bigcup_{t \in \cat L} \Supp(t) \subseteq \Spc(\cat T^c).
\]
\end{Rec}

\begin{Not}
For any subset $S \subseteq \Spc(\cat T^c)$ we write
\[
\cat T_S \coloneqq \SET{t \in \cat T}{\Supp(t) \subseteq S}
\]
for the localizing ideal of objects supported on $S$.
\end{Not}

\begin{Def}
The category $\cat T$ is said to be \emph{stratified} if the map
\[
\big\{ \text{localizing ideals of $\cat T$}\big\} \to \big\{\text{subsets of $\Spc(\cat T^c)$}\big\} 
\]
sending $\cat L$ to $\Supp(\cat L)$ is a bijection. The inverse is given by sending $S$ to $\cat T_S$.
\end{Def}

\begin{Rem}
When $\cat T$ is stratified, the localizing ideal $\cat T_S$ is generated by the set of objects $\SET{g_{\cat P} \in \cat T}{\cat P \in S}$. Hence, $\cat T_S$ is well generated; see \cite[Remark~A.9]{DellAmbrogioMartos24}, for example. If $\cat T = \Ho(\cat C)$ then the $\infty$-category $\cat C_S$ underlying $\cat T_S$ is presentable, in view of \cref{rem:term}.
\end{Rem}

\begin{Rem}
Suppose that $\cat T$ is stratified and let $S_2 \subseteq S_1$ be subsets of $\Spc(\cat T^c)$. Since $\cat T_{S_1}$ and $\cat T_{S_2}$ are well generated, the Verdier quotient $q \colon \cat T_{S_1} \to \cat T_{S_1} / \cat T_{S_2}$ is a Bousfield localization. This follows from the results of \cite{Neeman01}; see also \cite[Theorem~2.13]{BCHS25}. Therefore, we have
\begin{equation}\label{eq:right-ortho}
\cat T_{S_1} / \cat T_{S_2} \simeq \SET{t \in \cat T_{S_1}}{\Hom_{\cat T}(x, t) = 0 \text{ for all } x \in \cat T_{S_2}} = \cat T_{S_1} \cap \cat T_{S_2}^{\perp}.
\end{equation}
If $\cat T = \Ho(\cat C)$ then the Verdier quotient $\cat C_{S_1} / \cat C_{S_2}$ can be defined as the cofiber of the inclusion $\cat C_{S_2} \hook \cat C_{S_1}$ in $\Prst$. Hence $\cat T_{S_1} / \cat T_{S_2} \simeq \Ho(\cat C_{S_1} / \cat C_{S_2})$. See \cite[Section~5]{BlumbergGepnerTabuada13} for further details.
\end{Rem}

\begin{Exa}[Finite localization]\label{exa:finite-local}
Suppose that $\cat T$ is stratified with $\Spc(\cat T^c)$ noetherian. The specialization closed subsets of $\Spc(\cat T^c)$ correspond to the localizing ideals of $\cat T$ generated by compact objects of $\cat T$. Indeed, for any specialization closed subset $Y$, the localizing ideal $\cat T_Y$ is generated by $\cat T_Y^c \coloneqq \SET{x \in \cat T^c}{\supp(x) \subseteq Y}$. The associated Bousfield localization
\[
\cat T_Y \xhook{i} \cat T \xtwohead{q} \cat T / \cat T_Y
\]
is called a finite localization. The Neeman--Thomason theorem (\cite[Theorem~2.1]{Neeman96}) asserts that the induced functor on compact objects $\cat T^c \to (\cat T / \cat T_Y)^c$ is essentially surjective up to direct summands. Indeed, if $x \in (\cat T / \cat T_Y)^c$ then there exists some $a \in \cat T^c$ such that $x \oplus \Sigma x \cong q(a)$. It follows that the functor $\cat T^c \to (\cat T / \cat T_Y)^c$ induces an embedding
\[
\varphi \colon \Spc((\cat T / \cat T_{Y})^c) \hook \Spc(\cat T^c)
\]
on Balmer spectra whose image is the complement of $Y$. Moreover, by \cite[Lemma~2.13]{BarthelHeardSanders23b} we have
\begin{equation}\label{eq:supp-finite-loc}
\Supp(q^Rq(t)) = Y^c \cap \Supp(t) \quad \text{and} \quad \Supp(i^R i(t)) = Y \cap \Supp(t)
\end{equation}
where $q^R$ denotes the fully faithful right adjoint of $q$ and $i^R$ the right adjoint of $i$. It then follows from stratification that
\begin{equation}\label{eq:right-ortho-y}
\cat T / \cat T_Y \simeq \cat T_Y^{\perp} \simeq \im q^R = \cat T_{Y^c}.
\end{equation}
More generally:
\end{Exa}

\begin{Lem}\label{lem:quotient}
Suppose that $\cat T$ is stratified with $\Spc(\cat T^c)$ noetherian. Let $S_2 \subseteq S_1$ be subsets of $\Spc(\cat T^c)$. If $S_2$ is specialization closed then we have
\[
\cat T_{S_1} / \cat T_{S_2} \simeq \cat T_{S_1\setminus S_2}.
\]
\end{Lem}

\begin{proof}
Observe that
\[
\cat T_{S_1} / \cat T_{S_2} \simeq \cat T_{S_1} \cap \cat T_{S_2}^{\perp} = \cat T_{S_1} \cap (\cat T \cap \cat T_{S_2}^{\perp}) \simeq \cat T_{S_1} \cap \cat T / \cat T_{S_2} \simeq \cat T_{S_1} \cap \cat T_{S_2^c} = \cat T_{S_1\setminus S_2}
\]
where the first and second equivalence uses \eqref{eq:right-ortho} and the third equivalence is due to \eqref{eq:right-ortho-y}.
\end{proof}

\begin{Exa}[Local categories]\label{exa:local-cat}
Suppose that $\cat T$ is stratified with $\Spc(\cat T^c)$ noetherian. For $\cat P \in \Spc(\cat T^c)$, we denote by $Y_{\cat P}$ the complement of $\gen(\cat P)$, the set of generalizations of $\cat P$. Note that $Y_{\cat P}$ is the largest specialization closed subset not containing $\cat P$. By \cref{exa:finite-local} the finite localization $\cat T \to \cat T / \cat T_{Y_{\cat P}} \eqqcolon \cat T_{\cat P}$ induces an embedding $\varphi \colon \Spc(\cat T_{\cat P}^c) \hook \Spc(\cat T^c)$ on Balmer spectra whose image is $\gen(\cat P)$. Hence, the category $\cat T_{\cat P}$ is \emph{local} in the sense that the spetrum $\Spc(\cat T_{\cat P}^c)$ admits a unique closed point --- see \cite[Definition~1.25]{BarthelHeardSanders23b}. For any object $t \in \cat T$ we will write $t_{\cat P}$ for its image in $\cat T_{\cat P}$.
\end{Exa}

\begin{Rem}
An observation due to Efimov \cite[Proposition~1.18]{Efimov24} is that the Neeman--Thomason localization theorem holds not only for compactly generated categories but also for the dualizable ones:
\end{Rem}

\begin{Thm}[Efimov]\label{thm:neeman--thomason}
Let $\cat C \in \Prst$ be dualizable and $\cat S$ a full stable subcategory of $\cat C^c$. The localization $q: \cat C \to \cat C / \Ind(\cat S)$ induces an equivalence between $(\cat C / \Ind(\cat S))^c$ and the idempotent completion of $\cat C^c / \cat S$. Moreover, for any compact object $x \in (\cat C / \Ind(\cat S))^c$, there exists a compact object $a \in \cat C^c$ such that $q(a) \cong x \oplus \Sigma x$.
\end{Thm}

\begin{Rem}
The rest of the section is devoted to recalling the notion of cohomological stratification.
\end{Rem}

\begin{Not}
Let $\cat T$ be a rigidly-compactly generated tt-category. We set
\[
\Hom^*_{\cat T}(a,b) \coloneqq \bigoplus_{i \in \bbZ} \Hom_{\cat T}(a, \Sigma^i b)
\]
for any objects $a, b \in \cat T$. Recall that $R_{\cat T} \coloneqq \End^*_{\cat T}(\unit) \coloneqq \Hom^*_{\cat T}(\unit, \unit)$ is a graded-commutative $\bbZ$-graded ring which canonically acts on $\cat T$, making each $\Hom^*_{\cat T}(a,b)$ a graded $R_{\cat T}$-module. For the rest of the paper, all commutative algebra constructions about $R_{\cat T}$ will be the graded version. For example, an ideal of $R_{\cat T}$ will always mean a homogeneous ideal and $\Spec(R_{\cat T})$ will denote the homogeneous Zariski spectrum of $R_{\cat T}$. By \cite{Balmer10a} there exists a natural continuous comparison map
\[
\rho_{\cat T} \colon \Spc(\cat T^c) \to \Spec(R_{\cat T}).
\]
\end{Not}

\begin{Def}\label{def:coh-strat}
Let $\cat T$ be a rigidly-compactly generated tt-category. We say that $\cat T$ is \emph{cohomologically stratified} if:
\begin{enumerate}
\item $\cat T$ is \emph{noetherian}, meaning that the ring $R_{\cat T} = \End^*_{\cat T}(\unit)$ is noetherian and the $R_{\cat T}$-module $\Hom^*_{\cat T}(c, d)$ is finitely generated for any compact objects $c, d \in \cat T^c$;
\item $\rho_{\cat T} \colon \Spc(\cat T^c) \to \Spec(R_{\cat T})$ is a homeomorphism;
\item $\cat T$ is stratified.
\end{enumerate}
\end{Def}

\begin{Rem}
As explained in \cite[Section~7]{BarthelHeardSanders23b}, $\cat T$ is cohomologically stratified in the sense above if and only if it is stratified by the canonical action of $R_{\cat T}$ in the sense of \cite{BensonIyengarKrause11b}. In this case, we have $\rho_{\cat T}(\Supp(t)) = \SuppBIK(t)$ for all $t \in \cat T$; see \cite[Theorem~9.3]{Zou23}. Here $\SuppBIK$ denotes the support in the sense of Benson--Iyengar--Krause. In the proof of our main theorem we will use some properties of $\SuppBIK$ established in \cite{BensonIyengarKrause08}.
\end{Rem}

\begin{Exa}\label{exa:dr-coh-strat}
The derived category $\Der(A)$ of a commutative noetherian ring is cohomologically stratified.
\end{Exa}

\begin{Def}\label{def:loc-coh-strat}
Let $\cat T$ be a rigidly-compactly generated tt-category. We say that $\cat T$ is \emph{locally cohomologically stratified} if the local category $\cat T_{\cat P}$ is cohomologically stratified for every $\cat P \in \Spc(\cat T^c)$.
\end{Def}

\begin{Exa}\label{exa:dx-loc-coh-strat}
Let $X$ be a noetherian scheme and $\Derqc(X)$ the derived category of complexes of $\cat O_X$-modules with quasi-coherent cohomology. Although $\Derqc(X)$ is rarely cohomologically stratified if $X$ is nonaffine, it is always locally cohomologically stratified. Indeed, for any $\cat P \in \Spc(\Derperf(X))$ we can choose an affine open neighborhood $U \cong \Spec(A)$ of $\cat P$ so $A$ is a commutative noetherian ring. It then follows from \cite[Remark~5.9, Proposition~1.32, and Example~1.36]{BarthelHeardSanders23b} that we have an equivalence
\[
\Derqc(X)_{\cat P} \simeq \Der(A_{\frakp})
\]
where $\frakp \in \Spec(A) \cong U$ is the prime corresponding to $\cat P$. Therefore, $\Derqc(X)_{\cat P}$ is cohomologically stratified since $\Der(A_{\frakp})$ is by \cref{exa:dr-coh-strat}.
\end{Exa}

\begin{Rem}
The example above shows that a locally cohomologically stratified category need not be cohomologically stratified. Nevertheless, if $\Spc(\cat T^c)$ is noetherian, then being locally cohomologically stratified implies being stratified:
\end{Rem}

\begin{Prop}\label{prop:local-coh-strat}
Let $\cat T$ be a rigidly-compactly generated tt-category. Consider the following statements:
\begin{enumerate}
\item $\cat T$ is cohomologically stratified;
\item $\cat T$ is locally cohomologically stratified and $\Spc(\cat T^c)$ is noetherian;
\item $\cat T$ is stratified and $\Spc(\cat T^c)$ is noetherian.
\end{enumerate}
We have (a) $\implies$ (b) $\implies$ (c).
\end{Prop}

\begin{proof}
For (b) $\implies$ (c), since $\Spc(\cat T^c)$ is noetherian, the local-to-global principle holds by \cite[Theorem~3.22]{BarthelHeardSanders23b}. Thus, to prove that $\cat T$ is stratified it suffices to show that $\cat T_{\cat P}$ satisfies the minimality condition at the unique closed point, by \cite[Corollary~5.3]{BarthelHeardSanders23b}. This follows from the fact that $\cat T_{\cat P}$ is stratified since it is cohomologically stratified by our assumption.

For (a) $\implies$ (b), since $\rho_{\cat T}$ is a homeomorphism, the localization $\cat T \to \cat T_{\cat P}$ at a prime $\cat P$ coincides with the algebraic localization at the prime $\frakp \coloneqq \rho_{\cat T}(\cat P)$ by \cite[Example~6.1]{Zou23}. Hence, by \cite[Construction~3.5 and Theorem~3.6]{Balmer10a} we have
\[
\Hom^*_{\cat T_{\cat P}}(x_{\cat P}, y_{\cat P}) \cong \Hom^*_{\cat T}(x, y)_{\frakp}
\]
for any compact objects $x, y \in \cat T^c$. In particular, $R_{\cat T_{\cat P}} \cong (R_{\cat T})_{\frakp}$ is noetherian. To check that $\cat T_{\cat P}$ is noetherian, suppose that $c, d$ are compact objects of $\cat T_{\cat P}$. By the Neeman--Thomason localization theorem there exist compact objects $a, b \in \cat T^c$ such that $a_{\cat P} = c \oplus \Sigma c$ and $b_{\cat P} = d \oplus \Sigma d$. It follows that
\[
\Hom^*_{\cat T_{\cat P}}(c \oplus \Sigma c, d \oplus \Sigma d) \cong \Hom^*_{\cat T}(a, b)_{\frakp}
\]
is finitely generated over $R_{\cat T_{\cat P}}$. Hence, $\Hom^*_{\cat T_{\cat P}}(c, d)$ is also finitely generated. Therefore, $\cat T_{\cat P}$ is noetherian. Since $R_{\cat T_{\cat P}}$ is noetherian, $\rho_{\cat T_{\cat P}}$ is surjective by \cite[Theorem~7.3]{Balmer10a}. Moreover, the following diagram commutes by the naturality of $\rho$ (\cite[Theorem~5.3(c)]{Balmer10a}):
\[
\begin{tikzcd}
\Spc(\cat T_{\cat P}^c) \arrow[r, two heads, "\rho_{\cat T_{\cat P}}"] \arrow[d, hook] & \Spec(R_{\cat T_{\cat P}}) \arrow[d, hook] \\
\Spc(\cat T^c) \arrow[r, "\sim", "\rho_{\cat T}"']                                            & \Spec(R_{\cat T}).
\end{tikzcd}
\]
It follows that $\rho_{\cat T_{\cat P}}$ is a bijection and hence a homeomorphism by \cite[Corollary~2.8]{Lau23}. It remains to show that stratification passes to local categories, but this was proved in \cite[Corollary~4.9]{BarthelHeardSanders23b}.
\end{proof}

\section{Convexity and dualizability}
In this section we recall the notion of convexity and prove our main theorem.

\begin{Hyp}
Throughout this section, $\cat C$ denotes a rigidly-compactly generated tt-$\infty$-category.
\end{Hyp}

\begin{Not}
For any subset $S \subseteq \Spc(\cat C^c)$, we write $\cat C_S = \SET{t \in \cat C}{\Supp(t) \subseteq S}$ for the localizing ideal of objects supported on $S$. For example, if $\cat C$ is local with unique closed point $\cat M$, then $\cat C_{\{\cat M\}}$ is the localizing ideal of objects supported at the closed point.
\end{Not}

\begin{Rec}
Let $X$ be any spectral space. The specialization order on $X$ is defined by $x \le y$ if and only if $y$ is a specialization of $x$, that is, $y \in \overline{\{x\}}$, or equivalently, $x \in \gen(y)$. Note that the specialization closed subsets are exactly the unions of closed subsets.
\end{Rec}

\begin{Def}
Let $X$ be a spectral space. A subset $S$ of $X$ is said to be \emph{convex} if it is convex with respect to the specialization order on $X$, that is
\[
x \le y \le z \text{ and } x, z \in S \implies y \in S.
\]
\end{Def}

\begin{Rem}
Specialization closed subsets are clearly convex. More generally, we have:
\end{Rem}

\begin{Lem}\label{lem:convex}
A subset $S$ of a spectral space $X$ is convex if and only if there exist specialization closed subsets $S_2 \subseteq S_1$ of $X$ such that $S = S_1 \setminus S_2$.
\end{Lem}

\begin{proof}
The if part is immediate by definition. For the only if part, we can take $S_1$ to be the specialization closure of $S$, that is, $\bigcup_{x \in S}\overline{\{x\}}$, and $S_2$ to be $S_1 \setminus S$. The convexity of $S$ guarantees that $S_2$ is specialization closed.
\end{proof}

\begin{Prop}\label{prop:dualizable}
Suppose that $\cat C$ is stratified with $\Spc(\cat C^c)$ noetherian. If $S$ is a convex subset of $\Spc(\cat C^c)$ then $\cat C_S$ is compactly generated and hence dualizable.
\end{Prop}

\begin{proof}
The first part of the proof in \cite[Theorem~3.10]{Efimov24} carries over, \emph{mutatis mutandis}. We spell out the detail for the convenience of the reader. By \cref{lem:convex} there exist some specialization closed subsets $S_1, S_2$ such that $S = S_1 \setminus S_2$. We thus have $\cat T_S \simeq \cat T_{S_1} / \cat T_{S_2}$ by \cref{lem:quotient}. The latter is compactly generated by \cite[Lemma~2.17]{BarthelHeardValenzuela18}, for example.
\end{proof}

\begin{Rem}
A subtle detail to keep in mind is that a localizing ideal $\cat C_S$ can be compactly generated (as a stable $\infty$-category) without being generated by compact objects in the ambient category $\cat C$. In other words, the inclusion $\cat C_S \hook \cat C$ need not preserve compact objects. \cref{prop:dualizable} demonstrates this. For this reason, we will not use the term ``compactly generated localizing ideal'' as used in \cite{BarthelHeardSanders23b} because their usage more precisely means ``generated by compact objects in $\cat C$''.
\end{Rem}

\begin{Thm}\label{thm:main}
If $\cat C$ is locally cohomologically stratified with noetherian Balmer spectrum $\Spc(\cat C^c)$ then there is an inclusion-preserving bijection
\[
\{\text{dualizable localizing ideals of }\cat C\} \xra{\sim} \{\text{convex subsets of }\Spc(\cat C^c)\}
\]
sending $\cat L$ to $\Supp(\cat L)$ with inverse sending $S$ to $\cat C_S$. Moreover, in this case, the dualizable localizing ideals are themselves compactly generated.
\end{Thm}

\begin{Rem}
To prove \cref{thm:main}, we need some preparation.
\end{Rem}

\begin{Not}
Let $\cat T$ be a rigidly-compactly generated tt-category. Recall that we have the graded ring $R \coloneqq R_{\cat T} = \End_{\cat T}^*(\unit)$ and an $R$-module $\Hom^*_{\cat T}(a, b)$ for any objects $a, b \in \cat T$. For any $n \in \bbZ$ and any $R$-module $M$, we write $M[n]$ for the degree shift by $n$. For example, we have
\[
\Hom^*_{\cat T}(a,\Sigma^n b) \cong \Hom^*_{\cat T}(a,b)[n].
\]
We will also write $H^*_{a}(b)$ for $\Hom_{\cat T}^*(a, b)$ and think of it as the cohomology of $b$ with respect to $a$.
\end{Not}

\begin{Lem}\label{lem:matlis-lift}
Let $\cat T$ be a rigidly-compactly generated tt-category. Suppose that $I$ is an injective $R$-module. There exists an object $I_{\unit} \in \cat T$ such that for any set $\{n_i\}_i$ of integers there is a natural isomorphism
\[
\alpha \colon \Hom_{\cat T}(-, \coprod_i \Sigma^{n_i} I_{\unit}) \to \Hom_{R}(H^*_{\unit}(-), \bigoplus_i I[n_i]).
\]
\end{Lem}

\begin{proof}
Since $I$ is injective, Brown representability yields an object $I_{\unit} \in \cat T$ and a natural isomorphism
\[
\Phi \colon \Hom_{\cat T}(-,I_{\unit}) \xra{\sim} \Hom_{R}(H^*_{\unit}(-),I)
\]
which, by a degree shifting argument, extends to a natural isomorphism of $R$-modules
\begin{equation}\label{eq:phi-star}
\Phi^* \colon \Hom^*_{\cat T}(-,I_{\unit}) \xra{\sim} \Hom^*_R(H^*_{\unit}(-),I).
\end{equation}
Observe that
\begin{equation}\label{eq:phi}
\Phi^n_{\Sigma^n I_{\unit}}(\id_{\Sigma^n I_{\unit}}) = \Phi_{I_{\unit}}(\id_{I_{\unit}})[n]
\end{equation}
for any $n \in \bbZ$.

Setting $u \coloneqq \Phi_{I_{\unit}}(\id_{I_{\unit}}) \colon H^*_{\unit}(I_{\unit}) \to I$, we note that by the Yoneda lemma the map
\begin{equation}\label{eq:v}
v \colon H^*_{\unit}(\coprod_{i}\Sigma^{n_i}I_{\unit}) \cong \bigoplus_{i} H^*_{\unit}(\Sigma^{n_i}I_{\unit}) \xra{\bigoplus_{i}u[n_i]} \bigoplus_{i}I[n_i]
\end{equation}
gives rise to a natural transformation
\[
\alpha \colon \Hom_{\cat T}(-,\coprod_{i} \Sigma^{n_i} I_{\unit}) \to \Hom_{R}(H^*_{\unit}(-), \bigoplus_{i}I[n_i]).
\]
To prove that $\alpha$ is an isomorphism. It suffices to show that $\alpha_c$ is an isomorphism for any compact object $c \in \cat T^c$. Indeed, we will show that for any $c \in \cat T^c$, the following diagram commutes:
\[
\begin{tikzcd}
\bigoplus_{i} \Hom_{\cat T}(c,\Sigma^{n_i}I_{\unit}) \arrow[r, "\bigoplus_{i} \Phi^{n_i}_c", "\sim"'] \arrow[d, "\lambda_1"', "\rotatebox{90}{$\sim$}"] & \bigoplus_{i} \Hom_{R}(H^*_{\unit}(c), I[n_i]) \arrow[d, "\lambda_2", "\rotatebox{90}{$\sim$}"'] \\
\Hom_{\cat T}(c,\coprod_{i} \Sigma^{n_i} I_{\unit}) \arrow[r, "\alpha_c"]                         & \Hom_{R}(H^*_{\unit}(c), \bigoplus_{i}I[n_i])                       
\end{tikzcd}
\]
where $\lambda_1, \lambda_2$ are the canonical isomorphisms. It suffices to check that for any $k \ge 1$ the following diagram commutes:
\[
\begin{tikzcd}
\Hom_{\cat T}(c,\Sigma^{n_k}I_{\unit}) \arrow[r, "\Phi^{n_k}_c", "\sim"'] \arrow[d, hook, "\iota_1"'] & \Hom_{R}(H^*_{\unit}(c), I[n_k]) \arrow[d, hook', "\iota_2"] \\
\Hom_{\cat T}(c,\coprod_{i}\Sigma^{n_i}I_{\unit}) \arrow[r, "\alpha_c"]                         & \Hom_{R}(H^*_{\unit}(c), \bigoplus_{i}I[n_i])                       
\end{tikzcd}
\]
where $\iota_1$ and $\iota_2$ are the natural inclusions. Indeed, consider any map $h \colon c \to \Sigma^{n_k} I_{\unit}$. By definition, $\alpha_c= v \circ H^*_{\unit}(-)$. Hence, $\alpha_c \iota_1$ sends $h$ to the map
\begin{equation}\label{eq1}
H^*_{\unit}(c) \xra{H^*_{\unit}(h)} H^*_{\unit}(\Sigma^{n_k} I_{\unit}) \hook
H^*_{\unit}(\coprod_{i} \Sigma^{n_i} I_{\unit}) \xra{v} \bigoplus_{i}I[n_i].
\end{equation}
On the other hand, we have
\[
\Phi^{n_k}_c(h) = \Hom_R(h \circ \Phi^{n_k}_{\Sigma^{n_k} I_{\unit}}(\id_{\Sigma^{n_k} I_{\unit}}), I[n_k]) = \Hom_R(h \circ u[n_k], I[n_k])
\]
where the first equality holds by the naturality of $\Phi^{n_k}$ and the second is due to \eqref{eq:phi}. Hence, $\iota_2 \Phi^{n_k}_c$ sends $h$ to the map
\begin{equation}\label{eq2}
H^*_{\unit}(c) \xra{H^*_{\unit}(h)} H^*_{\unit}(\Sigma^{n_k} I_{\unit}) \xra{u[n_k]} I[n_k] \hook \bigoplus_{i} I[n_i].
\end{equation}
A routine diagram chase shows that \eqref{eq1} coincides with \eqref{eq2}, which finishes the proof.
\end{proof}

\begin{Lem}\label{lem:finite-generation}
Let $\cat T$ be a rigidly-compactly generated tt-category. Suppose that $\cat T$ is cohomologically stratified and that $\cat T$ is local with unique closed point $\cat M$. If $t$ is an object of $\cat T$ such that the functor $\Hom_{\cat T}(t,-)$ commutes with coproducts in $\cat T_{\{\cat M\}}$, then the $R$-module $H^*_{\unit}(t)$ is finitely generated.
\end{Lem}

\begin{proof}
By cohomological stratification the comparison map $\rho \colon \Spc(\cat T^c) \to \Spec(R)$ is a homeomorphism. It follows that the noetherian ring $R$ is local with unique maximal ideal $\frakm \coloneqq \rho(\cat M)$. We denote by $I$ the injective hull of $R / \frakm$. Therefore, the modules $\{I[n] \mid n \in \bbZ\}$ cogenerate the category of $R$-modules (\cf\cite[Theorem~19.8]{Lam99}). If $H^*_{\unit}(t)$ is not finitely generated then we can choose an infinite sequence of submodules
\[
0 = N_0 \subsetneq N_1 \subsetneq N_2 \subsetneq \cdots \subsetneq H^*_{\unit}(t).
\]
By cogeneration, for every $i \ge 1$ there exists a nonzero map $f_i \colon N_i \to I[n_i]$ for some~$n_i$ with $f_i|_{N_{i-1}}=0$. Write $N \coloneqq \bigcup_{i \ge 1}N_i$. By injectivity, we extend each $f_i$ to a map $g_i \colon N \to I[n_i]$ to obtain a map $N \to \prod_i I[n_i]$. The image of this map is contained in $\bigoplus_i I[n_i]$. Indeed, for any nonzero element $m \in N$, there exists some $k$ such that $m \notin N_{k-1}$ and $m \in N_k$, so $g_i$ annihilates $m$ whenever $i > k$. Therefore, we obtain a map $g \colon N \to \bigoplus_i I[n_i]$ which does not factor through any finite direct sum since every $g_i$ is nonzero. Since $R$ is noetherian, the target of $g$ is also injective. Hence, $g$ extends to a map $f \colon H^*_{\unit}(t) \to \bigoplus_i I[n_i]$ which does not factor through any finite direct sum since $g$ does not.

By \cref{lem:matlis-lift} we obtain an isomorphism
\[
\alpha_t \colon \Hom_{\cat T}(t,\coprod_i \Sigma^{n_i} I_{\unit}) \xra{\sim} \Hom_{R}(H^*_{\unit}(t), \bigoplus_i I[n_i]).
\]
Let $\tilde{f} \colon t \to \coprod_i \Sigma^{n_i} I_{\unit}$ be the map corresponding to the map $f \colon H^*_{\unit}(t) \to \bigoplus_{i} I[n_i]$. Hence, $f = \alpha_t(\tilde{f}) = v \circ H^*_{\unit}(\tilde{f})$ where $v \colon H^*_{\unit}(\coprod_{i}\Sigma^{n_i}I_{\unit}) \to \bigoplus_{i}I[n_i]$ comes from \eqref{eq:v}. We claim that $\tilde{f}$ does not factor through any finite coproduct. Otherwise, there would exist a map $\tilde{f}' \colon t \to \coprod_{j} \Sigma^{n_j} I_{\unit}$ where $j$ ranges over a finite set of positive integers such that the following diagram commutes:
\[
\begin{tikzcd}
t \arrow[r, "\tilde{f}'"'] \arrow[rr, "\tilde{f}", bend left=15] & \coprod_j \Sigma^{n_j} I_{\unit} \arrow[r, hook] & \coprod_i \Sigma^{n_i} I_{\unit}.
\end{tikzcd}
\]
Applying \cref{lem:matlis-lift} to the integers $\{n_j\}_j$, we obtain an isomorphism
\[
\alpha'_t \colon \Hom_{\cat T}(t,\coprod_j \Sigma^{n_j} I_{\unit}) \xra{\sim} \Hom_{R}(H^*_{\unit}(t), \bigoplus_j I[n_j])
\]
such that $\alpha'_t= v' \circ H^*_{\unit}(-)$ where $v' \colon H^*_{\unit}(\coprod_j \Sigma^{n_j} I_{\unit}) \to \bigoplus_j I[n_j]$ is the map \eqref{eq:v} with $i$ replaced by $j$. Setting $f' = \alpha'_t(\tilde{f}') = v' \circ H^*_{\unit}(\tilde{f}')$ we observe that the diagram below commutes:
\[
\begin{tikzcd}
H^*_{\unit}(t) \arrow[r, "H^*_{\unit}(\tilde{f}')"] \arrow[rr, "H^*_{\unit}(\tilde{f})", bend left=35] \arrow[rrr, "f", bend left=49] \arrow[rrd, "f'"'] & H^*_{\unit}(\coprod_j \Sigma^{n_j} I_{\unit}) \arrow[r, hook] \arrow[rd, "v'"] & H^*_{\unit}(\coprod_i \Sigma^{n_i} I_{\unit}). \arrow[r, "v"] & {{\bigoplus_i I[n_i]}.} \\
                                                                                                                                                         &                                                                                & {\bigoplus_j I[n_j]} \arrow[ru, hook]                         &                        
\end{tikzcd}
\]
However, this leads to a contradiction since we showed that $f$ does not factor through any finite direct sum, so the claim follows. That is, $\tilde{f} \colon t \to \coprod_i \Sigma^{n_i} I_{\unit}$ does not factor through any finite coproduct. This will contradict the hypothesis that $\Hom_{\cat T}(t,-)$ commutes with coproducts in $\cat T_{\{\cat M\}}$ if we can show $I_{\unit} \in \cat T_{\{\cat M\}}$. Indeed, by \eqref{eq:phi-star} we have $H^*_{c}(I_{\unit}) \cong \Hom^*_R(H^*_{\unit}(c),I)$ for any $c \in \cat T^c$. The latter is $\frakm$-torsion since $I$ is $\frakm$-torsion and $H^*_{\unit}(c)$ is finitely generated by the noetherian assumption (\cref{def:coh-strat}). It then follows from \cite[Lemma~2.4(2) and Corollary~5.3]{BensonIyengarKrause08} that $\SuppBIK(I_{\unit}) \subseteq \{\frakm\}$, which is equivalent to $\Supp(I_{\unit}) \subseteq \{\cat M\}$ by \cite[Theorem~9.3]{Zou23}. Therefore $I_{\unit} \in \cat T_{\{\cat M\}}$, which completes the proof.
\end{proof}

\begin{Rem}
In the case that $\cat T$ is \emph{unigenic}, meaning that it is generated by the unit $\unit$, the lemma above has a much shorter proof by utilizing \cite[Lemma~2.1]{Hovey07}. We leave the details to the interested reader.
\end{Rem}

\begin{Cor}\label{cor:special-closed}
In the situation of \cref{lem:finite-generation}, the support $\Supp(t)$ is a specialization closed subset of $\Spc(\cat T^c)$.
\end{Cor}

\begin{proof}
First we claim that $H^*_{c}(t)$ is finitely generated for any $c \in \cat T^c$. Indeed, since we have $H^*_{c^{\vee} \otimes t}(-) \cong H^*_t(c \otimes -)$ by adjunction, $c^{\vee} \otimes t$ also satisfies the condition in \cref{lem:finite-generation}, so $H^*_{c}(t) \cong H^*_{\unit}(c^{\vee} \otimes t)$ is finitely generated. It then follows from \cite[Corollary~5.3 and Lemma~2.2(1)]{BensonIyengarKrause08} that
\[
\SuppBIK(t) \subseteq \bigcup_{c \in \cat T^c} \cal V(\Ann_R H^*_{c}(t))
\]
Moreover, by \cite[Theorem~5.5 and Lemma~2.2(1)]{BensonIyengarKrause08} we have
\[
\bigcup_{c \in \cat T^c} \cal V(\Ann_R H^*_{c}(t)) \subseteq \SuppBIK(t).
\]
Hence, $\SuppBIK(t) = \bigcup_{c \in \cat T^c} \cal V(\Ann_R H^*_{c}(t))$ is specialization closed. Since the comparison map $\rho$ is a homeomorphism by our assumption, we invoke \cite[Theorem~9.3]{Zou23} to conclude that $\Supp(t) = \rho \inv (\SuppBIK(t))$ is also specialization closed.
\end{proof}

\begin{Lem}\label{lem:notdualizable}
Suppose that $\cat C$ is local and cohomologically stratified. Let $S$ be a subset of $\Spc(\cat C^c)$ which contains the unique closed point $\cat M$ and which also contains a point $\cat Q$ such that $\overline{\{\cat Q\}} \nsubseteq S$ and $S \cap \overline{\{\cat Q\}} \setminus \{\cat Q\}$ is specialization closed. Then the localizing ideal $\cat C_{S \cap \overline{\{\cat Q\}}}$ is not dualizable.
\end{Lem}

\begin{proof}
Write $S_1$ for $S \cap \overline{\{\cat Q\}}$ and $S_2$ for $S \cap \overline{\{\cat Q\}} \setminus \{\cat Q\}$ and consider the localization sequence
\[
\cat C_{S_2} \xhook{i} \cat C_{S_1} \xtwohead{q} \cat C_{S_1} / \cat C_{S_2}.
\]
Since $S_2$ is specialization closed, \cref{lem:quotient} implies
\[
\im q^R \simeq \cat C_{S_1} / \cat C_{S_2} \simeq \cat C_{S_1 \setminus S_2} = \cat C_{\{\cat Q\}}.
\]
Singletons are convex, so $\cat C_{S_1} / \cat C_{S_2}$ is compactly generated by \cref{prop:dualizable}. Suppose \emph{ab absurdo} that $\cat C_{S_1}$ is dualizable. In view of \cref{exa:finite-local}, $\cat C_{S_2}$ is generated by a set of compact objects of $\cat T$ and thus they are also compact in $\cat C_{S_1}$. Hence, we can summon the Neeman--Thomason localization theorem (\cref{thm:neeman--thomason}): Choosing any nonzero compact object $y$ of $\cat C_{S_1} / \cat C_{S_2}$, we find a compact object $t$ of $\cat C_{S_1}$ such that $q(t) \cong y \oplus \Sigma y$. Now consider the cofiber sequence
\[
ii^R(t) \to t \to q^R q(t).
\]
Note that $\Supp(q^R q(t)) \subseteq \{\cat Q\}$ and $\Supp(ii^R(t)) \subseteq S_2$. Since $q(t)$ is nonzero, we have $\Supp(q^R q(t)) = \{\cat Q\}$ by stratification. It then follows from \cite[Proposition~7.17(e)]{BalmerFavi11} that $\cat Q \in \Supp(t)$. Moreover, since $t$ is compact in $\cat C_{S_1}$ and $S_1$ contains $\cat M$, we see that $t$ satisfies the condition in \cref{lem:finite-generation} and hence $\Supp(t)$ is specialization closed by \cref{cor:special-closed}. This implies $\overline{\{\cat Q\}} \subseteq \Supp(t) \subseteq S$, which contradicts our hypothesis.
\end{proof}

\begin{proof}[Proof of \cref{thm:main}]
Since $\cat C$ is locally cohomologically stratified and $\Spc(\cat C^c)$ is noetherian, $\cat C$ is stratified by \cref{prop:local-coh-strat}. Therefore, in view of \cref{prop:dualizable}, it suffices to show that if $\cat C_S$ is dualizable then $S$ is convex. To this end, let $S$ be a nonconvex subset of $\Spc(\cat C^c)$ such that $\cat C_S$ is dualizable. By nonconvexity, there exist two points $\cat Q, \cat P \in S$ such that $\overline{\{\cat Q\}} \cap \gen(\cat P) \nsubseteq S$. Consider the set $\SET{\cat R \in S}{\overline{\{\cat R\}} \cap \gen(\cat P) \nsubseteq S}$. Being a nonempty subset of the noetherian space $\Spc(\cat C^c)$, it admits an element that is maximal with respect to the specialization order, which we still denote by $\cat Q$. Hence, we have
\[
\overline{\{\cat Q\}} \cap \gen(\cat P) \nsubseteq S \text{ and } \forall \cat R \in S \cap \overline{\{\cat Q\}} \setminus \{\cat Q\}\text{: }\overline{\{\cat R\}} \cap \gen(\cat P) \subseteq S.
\]
Consider the finite localization $q \colon \cat C \to \cat C_{\cat P} \eqqcolon \cat D$ at the prime $\cat P$ (recall \cref{exa:local-cat}). The associated map on spectra is an embedding $\varphi \colon \Spc(\cat D^c) \hook \Spc(\cat C^c)$ whose image is $\gen(\cat P)$. Writing $S'$ for $\varphi \inv (S)$ and $\cat Q'$ for $\varphi \inv (\cat Q)$, and working in $\Spc(\cat D^c)$, we have
\[
\overline{\{\cat Q'\}} \nsubseteq S' \text{ and } \forall \cat R' \in S' \cap \overline{\{\cat Q'\}} \setminus \{\cat Q'\}\text{: } \overline{\{\cat R'\}} \subseteq S'.
\]
Note that $\cat D$ is cohomologically stratified since $\cat C$ is locally cohomologically stratified. Hence, $\cat D$, $S'$, and $\cat Q'$ satisfy the conditions in \cref{lem:notdualizable}. It follows that $\cat D_{S' \cap \overline{\{\cat Q'\}}}$ is not dualizable.

On the other hand, applying the functor $-\otimes_{\cat C} \cat D$ to the inclusion $\cat C_S \hook \cat C$, by \cref{lem:fully-faithful} we obtain a fully faithful functor
\[
\cat C_{S} \otimes_{\cat C} \cat D \hook \cat C \otimes_{\cat C} \cat D \simeq \cat D
\]
with essential image $q(\cat C_S)$. Note that
\[
\Supp(q(\cat C_S)) = \varphi \inv (\Supp(\cat C_S)) = \varphi \inv (S) = S'
\]
where the first equality follows from \cite[Corollary~5.30]{Zou23} and the second is due to the stratification of $\cat C$. This tells us $\cat C_{S} \otimes_{\cat C} \cat D \simeq q(\cat C_S) = \cat D_{S'}$ by the stratification of~$\cat D$. \cref{lem:dualizability} then implies that $\cat D_{S'}$ is dualizable. Now consider the colocalization $i^R \colon \cat D \to \cat D_{\overline{\{\cat Q'\}}}$. By \cref{lem:fully-faithful} again, we obtain a fully faithful functor
\[
\cat D_{S'} \otimes_{\cat D} \cat D_{\overline{\{\cat Q'\}}} \hook \cat D \otimes_{\cat D} \cat D_{\overline{\{\cat Q'\}}} \simeq \cat D_{\overline{\{\cat Q'\}}}
\]
with essential image $i^R(\cat D_{S'})$. Note that
\[
\Supp(ii^R(\cat D_{S'})) = \overline{\{\cat Q'\}} \cap \Supp(\cat D_{S'}) = \overline{\{\cat Q'\}} \cap S'
\]
where the first equality is by \eqref{eq:supp-finite-loc} and the second by the stratification of $\cat D$. We thus have $\cat D_{S'} \otimes_{\cat D} \cat D_{\overline{\{\cat Q'\}}} \simeq i^R(\cat D_{S'}) = \cat D_{S'\cap\overline{\{\cat Q'\}}}$ by the stratification of $\cat D$ again. Invoking \cref{lem:dualizability} one more time, we see that $\cat D_{S'\cap\overline{\{\cat Q'\}}}$ is dualizable, which is absurd.
\end{proof}

\begin{Rem}\label{rem:proof}
Our proof of \cref{thm:main} is inspired by \cite[Theorem~3.10]{Efimov24} where Efimov establishes a one-to-one correspondence between the dualizable localizing ideals of the derived category $\Der(A)$ of any commutative noetherian ring $A$ and the convex subsets of the spectrum $\Spec(A)$. A key result in Efimov’s proof is \cite[Lemma~3.12]{Efimov24} which shows that in the case $\cat C = \Der(A)$, the object $t$ in \cref{lem:finite-generation} is compact in $\cat C$. This is a strictly stronger statement than our conclusion in \cref{cor:special-closed} that the support of $t$ is specialization closed in $\Spc(\cat C^c)$. However, the latter is sufficient to establish the desired \cref{lem:notdualizable} which corresponds to \cite[Lemma~3.13]{Efimov24}. Another important result due to Efimov is the the extension of the Neeman--Thomason localization theorem from compactly generated categories to dualizable categories (\cref{thm:neeman--thomason}). In the absence of this generalization, the proof of \cref{thm:main} only establishes a one-to-one correspondence between the convex subsets of $\Spc(\cat C^c)$ and the localizing ideals of $\cat C$ which are compactly generated.
\end{Rem}

\begin{Rem}
As explained in \cite[Remark~3.14]{Efimov24}, the classification of dualizable localizing ideals via convex subsets can fail for derived categories of non-noetherian rings. Nevertheless, it does hold in some cases. For example, consider the derived category $\Der(A)$ of any non-noetherian semi-artinian absolutely flat ring. It is stratified by \cite[Theorem~6.3]{Stevenson17}. Moreover, every subset of $\Spec(A)$ is specialization closed since $A$ has Krull dimension 0. Therefore, in light of \cref{exa:finite-local}, every localizing ideal of $\Der(A)$ is generated by a set of compact objects. In particular, all localizing ideals of $\Der(A)$ are dualizable and all subsets of~$\Spec(A)$ are convex.
\end{Rem}

\begin{Cor}
If $\cat C$ is cohomologically stratified then we have a bijection
\[
\{\text{dualizable localizing ideals of }\cat C\} \xra{\sim} \{\text{convex subsets of }\Spc(\cat C^c)\}
\]
\end{Cor}

\begin{proof}
This follows from \cref{thm:main} by \cref{prop:local-coh-strat}.
\end{proof}

\begin{Exa}
Cohomologically stratified categories abound. The reader is invited to consult \cite{BensonIyengarKrause11a, BensonIyengarKrause11b, DellAmbrogioStanley16, BensonIyengarKrausePevtsova18, Barthel21, Barthel22, Lau23, BIKP24, BCHNPS24, BBIKP25, BCHNP25} for many such examples. We end this paper with an example which is not cohomologically stratifed:
\end{Exa}

\begin{Exa}
The derived category $\Derqc(X)$ of any noetherian scheme is locally cohomologically stratified (\cref{exa:dx-loc-coh-strat}). Moreover, $\Spc(\Derperf(X)) \cong X$ is noetherian. Hence, \cref{thm:main} applies.
\end{Exa}

\raggedbottom 

\bibliographystyle{alpha}\bibliography{convex}

\end{document}